\documentclass[a4paper,11pt,twoside]{amsart}

\usepackage{amsmath}
\usepackage{amsthm}
\usepackage{amsfonts}
\usepackage{amssymb}
\usepackage{mathrsfs}
\usepackage{graphicx}
\usepackage{enumerate}
\usepackage{url}

\usepackage{tikz}

\usetikzlibrary{arrows}
\usetikzlibrary{decorations.markings}
\usepackage{tkz-graph}
\GraphInit[vstyle = Classic]

\tikzset{VertexStyle/.style = {shape = circle,
	ball color     = black,
	text           = black,
	inner sep      = 2pt,
	outer sep      = 0pt,
	minimum size   = 5 pt
}}

\theoremstyle{plain}
\newtheorem{lemma}{Lemma}[section]
\newtheorem{prop}[lemma]{Proposition}

\theoremstyle{remark}

\theoremstyle{definition}
\newtheorem{definition}[lemma]{Definition}

\newcommand{\N}{\mathbb{N}}
\newcommand{\Z}{\mathbb{Z}}
\newcommand{\M}{\mathcal{M}}
\DeclareMathOperator*{\cartprod}{\Box}
\begin{document}

\title{Transmission of perfect trees and rooted powers of graphs}

	\author{Nicol\'as Cianci}
	\address{Facultad de Ciencias Exactas y Naturales \\ Universidad Nacional de Cuyo\\ Mendoza, Argentina.}
	\email{nicocian@gmail.com}
	
	\thanks{Research partially supported by grant M049 of SeCTyP, UNCuyo.}
	
	\begin{abstract}
		We give exact formulas for the transmission (i.e. the sum of all distances between vertices) of perfect trees and rooted powers of (connected finite) graphs. 
	\end{abstract}
	
	\maketitle
\section{Introduction} 
The transmission $\delta\left(G\right)$ of a connected graph $G$ is defined as the sum of all the distances between vertices of $G$. Transmission is a graph invariant that has been studied, for example, in \cite{buckley1990distance, doyle1977mean, entringer1976distance, parhami2013exact, vsoltes1991transmission, yeh1994sum}.

Our main interest in the study of transmission of graphs lies in its application as an indicator of the performance of networks in the context of Internet of Things. Indeed, suppose a network of $n$ devices is modeled by an undirected simple connected graph $G$ with vertices $\{1,\ldots,n\}$, each of which represents a single device of the network, and edges $\{i,j\}$ for each pair of devices $i$ and $j$ that are able to send data packages, or messages, to each other. Assuming a routing protocol that minimizes the total amount of sent messages (or \emph{hop count}) is being used, the expected amount of individual messages sent after some time $T$ under ideal conditions is equal to  
\[S=\sum\limits_{i,j}d\left(i,j\right)p\left(i,j\right)\] where $d\left(i,j\right)$ represents the distance between the vertices $i$ and $j$ and $p\left(i,j\right)$ is the expected amount of messages sent from device $i$ to device $j$ over that time. Now, if $p\left(i,j\right)$ is either unknown or assumed to be independent of $i$ and $j$ for $i\neq j$, then the expected amount of messages sent over time $T$ reduces to
\[
	S=pT\sum\limits_{i,j}d\left(i,j\right)=pT\delta\left(G\right),
\] for some constant $p$ that is independent of the topology of the network. Hence, the transmission of graphs allows us to compare the performance of networks with different topologies when the rate of sent messages between specific devices is unknown or assumed to be equal to some constant $p$ for every pair of different devices.       

In this article we compute the transmission of perfect $n$--ary trees and \emph{rooted powers} of graphs. These results will be used in a future article, which is currently in progress, in which we will compare the performance of different network topologies in the context of Internet of Things.
\section{Preliminaries}
Throughout this article, every graph will be a \emph{rooted finite undirected simple graph}. Namely, a graph $G$ will be a 3--uple $G=\left(V, E, v_0\right)$ where $V$ is a finite non-empty set, $E$ is a set of 2--element subsets of $V$ and $v_0\in V$. 

For a graph $G=\left(V,E,v_0\right)$, we write $V\left(G\right)$, $E\left(G\right)$ and $v_0\left(G\right)$ for $V$, $E$ and $v_0$, respectively.
As usual, the elements of $V$ and $E$ will be called the \emph{vertices} of $G$ and the \emph{edges} of $G$, respectively, and $v_0$ will be called the \emph{root} of $G$.

For $i,j\in V\left(G\right)$ we say that $i$ and $j$ are \emph{adjacent} vertices of $G$, and we write $i\sim_G j$, if $\{i,j\}\in E\left(G\right)$. If $v\in V\left(G\right)$, the \emph{degree} of $v$ is the number of vertices of $G$ that are adjacent to $v$.

The number $|V\left(G\right)|$ of vertices of $G$ will be denoted by $|G|$.

Given two vertices $i$ and $j$ of $G$ and a non-negative integer $n$, a \emph{path (of length $n$)} between $i$ and $j$ is a sequence $x_0,\ldots,x_n$ of vertices of $G$ such that 
\[i=x_0\sim_G x_1\sim_G\cdots\sim_G x_n=j.\]
The distance $d_{G}(i,j)$ between $i$ and $j$ in $G$ is the infimum of the set of non-negative integers $n$ such that there is a path of length $n$ between $i$ and $j$. When the graph $G$ is understood, the distance $d_{G}(i,j)$ will be simply denoted by $d(i,j)$.

We will say that $G$ is \emph{connected} if there is a path between $i$ and $j$ for every pair of vertices $i$ and $j$ of $G$. Equivalently, $G$ is connected if the distance between $i$ and $j$ is finite for every $i,j\in V\left(G\right)$. 

If $G$ is connected and $v\in V\left(G\right)$, the \emph{transmission}\footnote{The transmission of a vertex $v$ in $G$ is also called the \emph{status} of $v$. See \cite{buckley1990distance}.} of $v$ in $G$, which will be denoted by $\delta\left(v,G\right)$, is defined as the sum of the distances between $v$ and every vertex of $G$, that is, \[\delta\left(v,G\right)=\sum\limits_{j\in V\left(G\right)}d\left(v,j\right).\] 
The transmission $\delta\left(v_0\left(G\right),G\right)$ of the root of $G$ in $G$ will be denoted by $\delta_0\left(G\right)$.

The transmission $\delta\left(G\right)$ of $G$ is defined as the sum \[\delta\left(G\right)=\sum\limits_{i\in V\left(G\right)}\delta\left(i,G\right)=\sum\limits_{i,j\in V\left(G\right)}d\left(i,j\right).\]

It is clear that the expressions
\[\frac{\delta\left(G\right)}{|G|^{2}}\quad\text{and}\quad \frac{\delta\left(G\right)}{|G|\left(|G|-1\right)}\]
are the \emph{mean distance between vertices} and the \emph{mean distance between different vertices} of $G$, respectively. Hence, the transmission of graphs can be used to compute other indicators of network performance as well \cite{parhami2013exact}.

\begin{definition}
	Let $G$ and $H$ be two rooted graphs. The \emph{one-point union} $G\lor H$ of $G$ and $H$ is the graph obtained by identifying the roots of $G$ and $H$. Namely, the set of vertices of $G\lor H$ is the wedge sum of the pointed sets $\left(V\left(G\right),v_0\left(G\right)\right)$ and $\left(V\left(H\right),v_0\left(H\right)\right)$ and two vertices $x,y$ of $G\lor H$ are adjacent in $G\lor H$ if and only if 
	\begin{itemize}
		\item there exist representatives of $x$ and $y$ in $G$ that are adjacent in $G$, or
		\item there exist representatives of $x$ and $y$ in $H$ that are adjacent in $H$.
	\end{itemize}
	
	Without loss of generality, we can always assume that $v_0\left(G\right)=v_0=v_0\left(H\right)$ and $V\left(G\right)\cap V\left(H\right)=\{v_0\}$, in which case, $G\lor H$ is just the union of the graphs $G$ and $H$, that is, $G\lor H=\left(V\left(G\right)\cup V\left(H\right),E\left(G\right)\cup E\left(H\right),v_0\right)$. Under this assumption, it is clear that $x$ and $y$ are adjacent in $G\lor H$ if and only if 
	\begin{itemize}
		\item $x,y\in V\left(G\right)$ and $x\sim_G y$, or
		\item $x,y\in V\left(H\right)$ and $x\sim_H y$.
	\end{itemize}
	
	Moreover, any path in $G\lor H$ from a vertex of $H$ to a vertex of $G$ must include the root $v_0$. Hence, it is easy to see that 
	\[d_{G\lor H}(x,y)=\begin{cases}
		d_G(x,y)&\text{if $x,y\in V(G)$,}\\
		d_H(x,y)&\text{if $x,y\in V(H)$,}\\
		d_G\left(x,v_0\right)+d_H\left(v_0,y\right)&\text{if $x\in V(G)$ and $y\in V(H)$,}\\
		d_H\left(x,v_0\right)+d_G\left(v_0,y\right)&\text{if $x\in V(H)$ and $y\in V(G)$.}\\
	\end{cases}
	\]
\end{definition}

The following proposition is easy to obtain.
\begin{prop}\label{wedge}
	Let $G$ and $H$ be two connected rooted graphs. Then 
			\[\delta_0\left(G\lor H\right)=\delta_0\left(G\right)+\delta_0\left(H\right)\] and
		\[\delta\left(G\lor H\right)=\delta\left(G\right)+\delta\left(H\right)+2\left(|H|-1\right)\delta_0\left(G\right)+2\left(|G|-1\right)\delta_0\left(H\right).\]
\end{prop}
\begin{proof}
	We assume that $v_0\left(G\right)=v_0=v_0\left(H\right)$ and that $V\left(G\right)\cap V\left(H\right)=\{v_0\}$.
	
	The first equality is clear. On the other hand, we have that
	\begin{align*}
		\delta\left(G\lor H\right)&=\sum\limits_{g,g'\in G}d\left(g,g'\right)+\sum\limits_{h,h'\in H}d\left(h,h'\right)+2\sum\limits_{\substack{g\in G\\g\neq v_0}}\sum\limits_{\substack{h\in H\\h\neq v_0}}\left(d\left(g,v_0\right)+d\left(v_0,h\right)\right)=\\
		&=\delta\left(G\right)+\delta\left(H\right)+2\sum\limits_{\substack{h\in H\\h\neq v_0}}\sum\limits_{\substack{g\in G\\g\neq v_0}}d\left(g,v_0\right)+2\sum\limits_{\substack{g\in G\\g\neq v_0}}\sum\limits_{\substack{h\in H\\h\neq v_0}}d\left(h,v_0\right)=\\
		&=\delta\left(G\right)+\delta\left(H\right)+2\left(|H|-1\right)\delta_0\left(G\right)+2\left(|G|-1\right)\delta_0\left(H\right).\qedhere
	\end{align*}
\end{proof}

The one-point union of finite rooted graphs is an associative and commutative operation. Moreover, we can recursively define the one-point union of a finite collection $G_1,\ldots,G_n$ of rooted graphs as 
\[\bigvee_{i=1}^{1}G_i=G_1\text{\quad and \quad}\bigvee_{i=1}^{n}G_i=\left(\bigvee_{i=1}^{n-1}G_i\right)\lor G_n.\]

Using proposition \ref{wedge} and an inductive argument we obtain the following more general result.
\begin{prop}\label{wedge_n}
	Let $G_1,\ldots,G_n$ be connected rooted graphs. Then
		\[\delta_0\left(\bigvee\limits_{i=1}^{n} G_i\right)=\sum\limits_{i=1}^{n}\delta_0\left(G_i\right)\] and
	\[\delta\left(\bigvee\limits_{i=1}^{n} G_i\right)=\sum\limits_{i=1}^{n}\delta\left(G_i\right)+2\sum\limits_{i=1}^{n}\delta_0\left(G_i\right)\left(1-n+\sum\limits_{j\neq i}|G_j|\right).\]
\end{prop}
\begin{definition}
	Let $G$ and $H$ be two rooted graphs. The \emph{rooted product} $G\circ H$ of $G$ and $H$ is the graph with set of vertices $V\left(G\right)\times V\left(H\right)$ and root $\left(v_0\left(G\right),v_0\left(H\right)\right)$, where two vertices $\left(g,h\right)$ and $\left(g',h'\right)$ are adjacent if and only if either
	\begin{itemize}
		\item $h=h'=v_0\left(H\right)$ and $g\sim_G g'$, or
		\item $g=g'$ and $h\sim_H h'$.
	\end{itemize}
	
	It is clear that 
	\[
	d_{G\circ H}\left(\left(g,h\right),\left(g',h'\right)\right)=\left\{\begin{array}{lr}
		d_H\left(h,h'\right)&\text{if $g=g'$,}\\
		d_H\left(h,h_0\right)+d_G\left(g,g'\right)+d_H\left(h_0,h'\right)&\text{if $g\neq g'$,}\end{array}\right.
	\]
	for every $g,g'\in V\left(G\right)$ and $h,h'\in V\left(H\right)$, where $h_0=v_0(H)$.
\end{definition}

The following result can be found in \cite{yeh1994sum}.
\begin{prop}[{\cite[Theorem 5]{yeh1994sum}}]\label{rooted_product}
	Let $G$ and $H$ be connected rooted graphs. Then
	\[\delta\left(G\circ H\right)=|G|\delta\left(H\right)+2|G|\left(|G|-1\right)|H|\delta_0\left(H\right)+|H|^{2}\delta\left(G\right).\]
\end{prop}
\begin{proof}
	Let $h_0$ be the root of $H$.

	We have that
	\begin{align*}
		\delta\left(G\circ H\right)&=\sum\limits_{h,h'}\sum\limits_{g,g'}d\left(\left(g,h\right),\left(g',h'\right)\right)=\\
		&=\sum\limits_{h,h'}\left(\sum\limits_{g}d\left(h,h'\right)+\sum\limits_{g\neq g'}\left(d\left(h,h_0\right)+d\left(h',h_0\right)+d\left(g,g'\right)\right)\right)=\\
		&=\sum\limits_{g}\sum\limits_{h,h'}d\left(h,h'\right)+\sum\limits_{g\neq g'}\sum\limits_{h}\sum\limits_{h'}\left(d\left(h,h_0\right)+d\left(h',h_0\right)\right)+\sum\limits_{h,h'}\sum\limits_{g\neq g'}d\left(g,g'\right)=\\
		&=|G|\delta\left(H\right)+2|G|\left(|G|-1\right)|H|\delta_0\left(H\right)+|H|^{2}\delta\left(G\right)
	\end{align*}
where, in each of the previous sums, $g$, $g'$ range in $V\left(G\right)$ and $h$, $h'$ range in $V\left(H\right)$.
\end{proof}

For the sake of completeness, and since the main goal of this article is to provide theoretic tools that will allow us to compare the performance of networks of different sizes and topologies in the context of Internet of Things, we state some simple results about transmission of well-known families of graphs that are commonly used to model such networks.

The proofs of the following three propositions are straightforward and will be left to the reader.
\begin{prop}
	Let $K_n$ be a complete graph with $n$ vertices. Then 
	\[\delta\left(K_n\right)=n\left(n-1\right).\]
\end{prop}
\begin{prop}
	Let $C_n$ be a circular graph with $n$ vertices. Then
	\[\delta\left(C_n\right)=\left\{\begin{array}{lr}
	\frac{n^{3}-n}{4}&\text{if $n$ is odd,}\\
	\frac{n^{3}}{4}&\text{if $n$ is even.}
	\end{array}\right.\]
\end{prop}
\begin{prop}
	Let $S_n$ be the star graph with $n+1$ vertices, that is, $S_n$ is the complete bipartite graph $K_{1,n}$. Then,
	\[\delta\left(S_n\right)=2n^{2}.\]
\end{prop}
\begin{definition}
	Let $n\in\N$. Let $R=\left(R_1,\ldots,R_n\right)\in\N^{n}$, and, for $i=1,\ldots,n$, let $P_i$ be the path graph with $R_i$ vertices $1,2,\ldots,R_i$. We define the mesh graph $\M\left(R\right)$ as the cartesian product 
	\[\M\left(R\right)=\cartprod\limits_{i=1}^{n}P_i.\]
	
	If $x=\left(x_1,\ldots,x_n\right)$ and $y=\left(y_1,\ldots,y_n\right)$ are two vertices of $\M\left(R\right)$ then \[d_{\M(R)}\left(x,y\right)=\sum\limits_{i=1}^{n}|x_i-y_i|.\]
\end{definition}

The following result is already known and can be found in \cite{parhami2013exact}. 
\begin{prop}[{\cite[Section 2]{parhami2013exact}}]\label{prop_mesh}
	Let $n\in\N$ and let $R=\left(R_1,\ldots,R_n\right)\in\N^{n}$. Then 
	\[\delta\left(\M\left(R\right)\right)=\frac{1}{3}\left(\prod\limits_{i=1}^{n}R^{2}_i\right)\sum\limits_{i=1}^{n}\left(R_i-\frac{1}{R_i}\right).\]
\end{prop}
\begin{proof}
	For $k\in\N$ we have that 
	\begin{align*}
		\sum\limits_{i=1}^{k}\sum\limits_{j=1}^{k}|i-j|&=\sum\limits_{i=1}^{k}\left(\sum\limits_{j=1}^{i-1}\left(i-j\right)+\sum\limits_{j=i+1}^{k}\left(j-i\right)\right)=\\
		&=\sum\limits_{i=1}^{k}\left(\sum\limits_{j=1}^{i-1}\left(i-j\right)\right)+\sum\limits_{i=1}^{k}\left(\sum\limits_{j=i+1}^{k}\left(j-i\right)\right)=\\
		&=\sum\limits_{i=1}^{k}\left(\frac{\left(i-1\right)i}{2}\right)+\sum\limits_{i=1}^{k}\left(\frac{\left(i-1\right)i}{2}\right)
		=\sum\limits_{i=1}^{k}\left(i-1\right)i=\frac{k^{3}-k}{3}.
	\end{align*}

	It follows that
	\begin{align*}
			\delta\left(\M\left(R\right)\right)&=\sum\limits_{i_1=1}^{R_1}\cdots\sum\limits_{i_n=1}^{R_n}\sum\limits_{j_1=1}^{R_1}\cdots\sum\limits_{j_n=1}^{R_n}\sum\limits_{t=1}^{n}|i_t-j_t|=\sum\limits_{t=1}^{n}\left(\prod\limits_{s\neq t}R^{2}_s\right)\sum\limits_{i_t=1}^{R_t}\sum\limits_{j_t=1}^{R_t}|i_t-j_t|=\\
			&=\sum\limits_{t=1}^{n}\left(\prod\limits_{s=1}^{n}R^{2}_s\right)\frac{R_t^{3}-R_t}{3R^{2}_t}=\frac{1}{3}\left(\prod\limits_{s=1}^{n}R^{2}_s\right)\sum\limits_{t=1}^{n}\left(R_t-\frac{1}{R_t}\right).
	\end{align*}
\end{proof}
Alternatively, proposition \ref{prop_mesh} can be proved using Theorem 1 of \cite{yeh1994sum}.
\section{Main results}
In this section we show that the transmission of the perfect $n$--ary tree of depth $k$ is 
\[\delta\left(T_n^{k}\right)=\frac{2n^{k+1}}{\left(n-1\right)^{2}}\left(kn^{k+1}+k-2n\frac{n^{k}-1}{n-1}\right)\] for every $n\in \N$ and every $k\in\N_0$, and that the transmission of the $k$--fold rooted product $G^{k}$ of a rooted connected graph $G$ with itself, is
\[\delta\left(G^{k}\right)=n^{k-1}\left(\frac{n^{k}-1}{n-1}\right)\delta\left(G\right)+2n^{k}\left(\left(k-1\right)n^{k-1}-\frac{n^{k-1}-1}{n-1}\right)\delta_0\left(G\right)\] for every $k\in\N$, where $n=|G|$.

\subsection{Transmission of perfect trees}
We define the following simple construction on rooted graphs.
\begin{definition}
	Let $S=(\{0,1\},\{\{0,1\}\},1)$, that is, $S$ is the complete graph with vertices $0$ and $1$ and root $1$, and let $G$ be any rooted graph. For simplicity, we assume that $0,1\not\in V(G)$. We define the rooted graph $\widetilde{G}$ as 
	\[\widetilde{G}=\left(V\left(G\lor S\right),E\left(G\lor S\right),0\right).\]
	
	In other words, the rooted graph $\widetilde{G}$ has the same underlying graph as $G\lor S$ but its root is the vertex $0$ of $S$ instead of the vertex $1$.
\end{definition}

\begin{lemma}\label{widetilde}
	Let $G$ be a rooted connected graph. Then \[\delta_0\left(\widetilde{G}\right)=\delta_0\left(G\right)+|G|\] and \[\delta\left(\widetilde{G}\right)=\delta\left(G\right)+2\delta_0\left(G\right)+2|G|.\]
\end{lemma}
\begin{proof}
	Let $g_0=v_0(G)$.
	It is clear that $d_{\widetilde{G}}\left(0,g\right)=d_G\left(g_0,g\right)+1$ for every $g\in V\left(G\right)$. Thus,
	\[\delta_0\left(\widetilde{G}\right)=\sum\limits_{g\in V\left(\widetilde{G}\right)}d_{\widetilde{G}}\left(0,g\right)=\sum\limits_{g\in V\left(G\right)}d_{\widetilde{G}}\left(0,g\right)=\sum\limits_{g\in V\left(G\right)}\left(d_{G}\left(g_0,g\right)+1\right)=\delta_0\left(G\right)+|G|.\]
	
	On the other hand, we have that
	\begin{align*}
	\delta\left(\widetilde{G}\right)&=\delta\left(G\lor S\right)=\delta\left(G\right)+\delta\left(S\right)+2\left(|S|-1\right)\delta_0\left(G\right)+2\left(|G|-1\right)\delta_0\left(S\right)=\\&=\delta\left(G\right)+2\delta_0\left(G\right)+2|G|
\end{align*}
	by \ref{wedge}.
\end{proof}

\begin{definition}
	Let $n\in \N$. For $k\in\N_0$ we recursively define the \emph{perfect $n$-ary tree of depth $k$}, denoted by $T^{k}_n$, as follows.
	\begin{itemize}
		\item $T_n^{0}$ is the only possible graph with one vertex.
		\item For $k\in\N$, we define \[T_n^{k}=\bigvee\limits_{i=1}^{n}\widetilde{T^{k-1}_n}.\]  
	\end{itemize}
	It is easy to see that $|T_n^{k}|=\frac{n^{k+1}-1}{n-1}$ for every $k\in\N_0$.
\end{definition}

Transmission of trees has been studied, for example, in \cite{doyle1977mean,parhami2013exact}. Our next result is an exact formula for the transmission of perfect trees.
	\begin{prop}\label{prop_trees}
		Let $n\in\N$. Then
		\[\delta\left(T_n^{k}\right)=\frac{2n^{k+1}}{\left(n-1\right)^{2}}\left(kn^{k+1}+k-2n\frac{n^{k}-1}{n-1}\right)\] for every $k\in\N_0$.
	\end{prop}
	\begin{proof}
		By \ref{widetilde} \[\delta_0\left(\widetilde{T_n^{k}}\right)=\delta_0\left(T_n^{k}\right)+\frac{n^{k+1}-1}{n-1}\] for $k\in\N_0$. 
		By \ref{wedge_n}, it follows that 
		\[\delta_0\left(T_n^{k+1}\right)=n\delta_0\left(\widetilde{T_n^{k}}\right)=n\delta_0\left(T_n^{k}\right)+n\frac{n^{k+1}-1}{n-1}\] for every $k\in\N_0$.
		Since $\delta_0\left(T_n^{0}\right)=0$, the reader can verify by induction on $k$ that 
		\[\delta_0\left(T_n^{k}\right)=\frac{kn^{k+2}-\left(k+1\right)n^{k+1}+n}{\left(n-1\right)^{2}}\] and that
		\[\delta_0\left(\widetilde{T_n^{k}}\right)=\frac{\left(k+1\right)n^{k+2}-\left(k+2\right)n^{k+1}+1}{\left(n-1\right)^{2}}\] for every $k\in\N_0$.
				
		By \ref{wedge_n} we obtain that
		\[\delta\left(T_n^{k+1}\right)=\delta\left(\bigvee\limits_{i=1}^{n}\widetilde{T_n^{k}}\right)=n\delta\left(\widetilde{T_n^{k}}\right)+2n\delta_0\left(\widetilde{T_n^{k}}\right)\left(n^{k+1}-1\right)
		\] for every $k\in\N_0$.
		On the other hand, by \ref{widetilde} it follows that
		\[\delta\left(\widetilde{T_n^{k}}\right)=\delta\left(T_n^{k}\right)+2\frac{\left(k+1\right)n^{k+2}-\left(k+2\right)n^{k+1}+1}{\left(n-1\right)^{2}}\] and hence
		\[\delta\left(T_n^{k+1}\right)=n\delta\left(T_n^{k}\right)+2n^{k+2}\frac{\left(k+1\right)n^{k+2}-\left(k+2\right)n^{k+1}+1}{\left(n-1\right)^{2}}\] for every $k\in\N_0$.
		
		Again, the reader can verify by induction on $k$ that \[\delta\left(T_n^{k}\right)=\frac{2n^{k+1}}{\left(n-1\right)^{2}}\left(kn^{k+1}+k-2n\frac{n^{k}-1}{n-1}\right)\] for every $k\in\N_0$, as claimed.
	\end{proof}
	
	From the last proposition one obtains that the transmission of a perfect binary tree of depth $k$ is given by
		\[\delta\left(T_2^{k}\right)=2^{k+2}\left(\left(k-2\right)2^{k+1}+k+4\right).\] This result was previously obtained in \cite{parhami2013exact}.
	
	Next, we give a generating function for the sequence $\{\delta\left(T_n^{k}\right):k\in\N\}$ for every $n\geq 2$.
	\begin{prop}
		Let $n\in \N$, $n\geq 2$.
		Then, the sequence $\{\delta\left(T_n^{k+1}\right):k\in\N_0\}$ is generated by the function $g$ defined by 
		\[g\left(x\right)=\frac{2n^{2}}{\left(1-nx\right)^{2}\left(1-n^{2}x\right)^{2}}.\]
	\end{prop}
	\begin{proof}
		Let $u_n$ be the (bilateral) sequence defined by 
		\[u_n\left(j\right)=\left\{
			\begin{array}{ll}
				\frac{n^{j}-1}{n-1}&\text{if $j\in\N$,}\\
				0&\text{if $j\in\Z_{\leq 0}$.}
			\end{array}
		  \right.\]
		  By \ref{prop_trees},
		\begin{align*}
		\delta\left(T_n^{k+1}\right)&=\frac{2n^{k+2}}{\left(n-1\right)^{2}}\sum\limits_{j=1}^{k+1}\left(n^{j}-1\right)\left(n^{k+2-j}-1\right)=\\&=2n^{k+2}\sum\limits_{j=-\infty}^{\infty}u_n\left(j\right)u_n\left(k+2-j\right)=2n^{k+2}\left(u_n\ast u_n\right)\left(k+2\right)=\\
		&=2n^2n^k\left(u_n\ast u_n\right)\left(k+2\right)
		\end{align*}
		for every $k\in\N_0$, where, as usual, $u_n\ast u_n$ denotes the convolution of $u_n$ with itself.
				
	Note that, since the sequences $\{1:k\in\N_0\}$ and $\{n^{k}:k\in\N_0\}$ are generated by the functions defined by \[\frac{1}{1-x} \quad\text{ and }\quad\frac{1}{1-nx},\] respectively, then the sequence $\{u_n\left(k+1\right):k\in\N_0\}$ is generated by the function defined by
	\[\frac{n}{n-1}\left(\frac{1}{1-nx}\right)-\frac{1}{n-1}\left(\frac{1}{1-x}\right)=\frac{1}{\left(1-x\right)\left(1-nx\right)}.\]
	Thus, the sequence $\{\left(u_n\ast u_n\right)\left(k+2\right):k\in\N_0\}$ is generated by the function defined by
	\[\frac{1}{\left(1-x\right)^{2}\left(1-nx\right)^{2}}.\]
	By means of the substitution $x\mapsto nx$, one obtains that the sequence $\{\delta\left(T_n^{k+1}\right):k\in\N_0\}$ is generated by the function $g$ as claimed. 
	\end{proof}
	\subsection{Transmission of rooted powers of graphs}
	In this subsection, we define the rooted powers $G^{k}$ of a rooted graph $G$ and show that the transmission of $G^{k}$ can be expressed in terms of $k$, $|G|$, $\delta\left(G\right)$ and $\delta_0\left(G\right)$ for every connected rooted graph $G$. 
	\begin{definition}
	Let $G$ be a rooted graph and let $k\in \N$. We define the \emph{rooted $k$--th power of $G$}, which will be denoted as $G^{k}$, as the $k$--fold rooted product of $G$ with itself, that is, $G^{1}=G$ and $G^{k+1}=G^{k}\circ G$ for every $k\in\N$.
\end{definition}
\begin{definition}\label{def_polys}
	For $k\in\N$ we define the following polynomials in the variable $n$:
	\begin{itemize}
		\item $a_k\left(n\right)=n^{k-1}\frac{n^{k}-1}{n-1}$, and
		\item $b_k\left(n\right)=2\left(k-1\right)n^{2k-1}-2n^{k}\frac{n^{k-1}-1}{n-1}$.
	\end{itemize}
\end{definition}
\begin{lemma}\label{lemma_polys}
	The polynomials $a_k$ and $b_k$ defined in \ref{def_polys} can be recursively defined by:
	\begin{itemize}
		\item $a_1\left(n\right)=1$ and $a_{k+1}\left(n\right)=n^{k}+n^{2}a_k\left(n\right)$ for $k\in \N$, and
		\item $b_1\left(n\right)=0$ and $b_{k+1}\left(n\right)=2n^{k+1}\left(n^{k}-1\right)+n^{2}b_k\left(n\right)$ for $k\in\N$,
	\end{itemize}
	respectively.
\end{lemma}
\begin{proof}
	The result follows easily by induction on $k$.
\end{proof}

\begin{prop}
	Let $G$ be a connected rooted graph, let $n=|G|$ and let $k\in \N$. Then 
	\[\delta\left(G^{k}\right)=a_k\left(n\right)\delta\left(G\right)+b_k\left(n\right)\delta_0\left(G\right).\]
\end{prop}
\begin{proof}
	By \ref{rooted_product} it is clear that 
	\[\delta\left(G^{k+1}\right)=n^{k}\delta\left(G\right)+2n^{k+1}\left(n^{k}-1\right)\delta_0\left(G\right)+n^{2}\delta\left(G^{k}\right).\]
 
	 This means that if $\delta\left(G^{k}\right)=x\delta\left(G\right)+y\delta_0\left(G\right)$ then
	 \[\delta\left(G^{k+1}\right)=\left(n^{k}+n^{2}x\right)\delta\left(G\right)+\left(2n^{k+1}\left(n^{k}-1\right)+n^{2}y\right)\delta_0\left(G\right).\]
	Since $\delta\left(G\right)=a_1\left(n\right)\delta\left(G\right)+b_1\left(n\right)\delta_0\left(G\right)$, the result follows from \ref{lemma_polys} by an inductive argument.
\end{proof}
\bibliographystyle{acm}
\bibliography{references}
\end{document}